\documentclass{amsart}
\usepackage{expl3}
\usepackage{amssymb,amsmath,amsthm}
\usepackage{stmaryrd}
\usepackage[english]{babel}

\usepackage[normalem]{ulem}

\usepackage[all]{xy}

\usepackage{xcolor}
\usepackage{float}
\usepackage{mathtools}

\usepackage{hyperref}
\usepackage{color}
\definecolor{darkred}{rgb}{.5,0,0}
\definecolor{darkgreen}{rgb}{0,.5,0}
\definecolor{darkblue}{rgb}{0,0,.5}
\hypersetup{
    pdfencoding=unicode,
    colorlinks=true,
    linkcolor=darkred,
    citecolor=darkgreen,
    urlcolor=darkblue,
    bookmarks=false,
    pdftoolbar=true,
    pdfmenubar=true,
    pdffitwindow=true,
    pdfstartview={FitH},
    pdfcreator={},
    pdfproducer={},
    pdfnewwindow=true,
    pdftitle={Cm solutions of semialgebraic or definable equations},
    pdfauthor={Edward Bierstone, Jean-Baptiste Campesato, Pierre D. Milman},
}

\usepackage{stix}
\DeclareSymbolFont{calletters}{OMS}{cmsy}{m}{n}
\DeclareSymbolFontAlphabet{\mathcal}{calletters}
\usepackage{tgpagella}
\usepackage[T1]{fontenc}
\usepackage[utf8]{inputenc}

\newtheorem{thm}{Theorem}[section]
\newtheorem{lemma}[thm]{Lemma}
\newtheorem{prop}[thm]{Proposition}
\newtheorem{cor}[thm]{Corollary}
\theoremstyle{definition}
\newtheorem{defi}[thm]{Definition}
\newtheorem{defrmk}[thm]{Definition and Remark}
\newtheorem{Q}[thm]{Problem}

\theoremstyle{remark}
\newtheorem{rmk}[thm]{Remark}
\newtheorem{rmks}[thm]{Remarks}
\newtheorem{notation}[thm]{Notation}

\newcommand{\supp}{\mathrm{supp}\,}

\newcommand{\mon}{\mathrm{mon}}

\newcommand{\bdry}{\mathrm{bdry}\,}

\newcommand{\lex}{\mathrm{lex}}

\newcommand{\al}{{\alpha}}
\newcommand{\be}{{\beta}}

\newcommand{\De}{{\Delta}}
\newcommand{\ga}{{\gamma}}
\newcommand{\Ga}{{\Gamma}}

\newcommand{\la}{{\lambda}}
\newcommand{\La}{{\Lambda}}

\newcommand{\p}{{\partial}}
\newcommand{\s}{{\sigma}}

\newcommand{\vp}{{\varphi}}

\newcommand{\om}{{\omega}}

\newcommand{\IN}{{\mathbb N}}

\newcommand{\IR}{{\mathbb R}}

\newcommand{\IK}{{\mathbb K}}

\newcommand{\cA}{{\mathcal A}}

\newcommand{\cC}{{\mathcal C}}
\newcommand{\cD}{{\mathcal D}}

\newcommand{\cN}{{\mathcal N}}

\newcommand{\cR}{{\mathcal R}}

\newcommand{\cV}{{\mathcal V}}

\newcommand{\fm}{{\mathfrak m}}

\newcommand{\oV}{\overline{V}}

\newcommand{\tG}{{\widetilde G}}

\newcommand{\tT}{{\widetilde T}}

\newcommand{\hvp}{{\hat \vp}}

\newcommand{\ua}{\underline{a}}
\newcommand{\uc}{\underline{c}}
\newcommand{\ux}{\underline{x}}

\newcommand\R{\mathcal R}

\renewcommand{\int}{\mathrm{o}}

\newcommand\Ss{\mathscr S}
\newcommand\Ts{\mathscr T}
\renewcommand\P{\mathscr P}
\newcommand\Fs{\mathscr F}
\newcommand\Gs{\mathscr G}
\newcommand\Cs{\mathscr C}
\newcommand\Us{\mathscr U}
\renewcommand\a{{\underline{a}}}

\newcommand\rank{\operatorname{rank}}

\newcommand\dist{\operatorname{dist}}
\renewcommand\Im{\operatorname{Im}}
\newcommand\Ker{\operatorname{Ker}}

\newcommand{\llb}{{[\![}}
\newcommand{\rrb}{{]\!]}}

\begin{document}

\title[$\cC^m$ solutions of semialgebraic or definable equations]
{$\cC^m$ solutions of semialgebraic or definable equations}

\author[E.~Bierstone]{Edward Bierstone}
\author[J.-B.~Campesato]{Jean-Baptiste Campesato}
\author[P.D.~Milman]{Pierre D. Milman}
\address{University of Toronto, Department of Mathematics, 40 St. George Street, Toronto, ON, Canada M5S 2E4}
\email{bierston@math.utoronto.ca}
\email{campesat@math.utoronto.ca}
\email{milman@math.utoronto.ca}
\thanks{Research supported by NSERC Discovery Grants RGPIN-2017-06537 (Bierstone) and
RGPIN-2018-04445 (Milman).}
\thanks{The authors are indebted to Charles Fefferman for many illuminating discussions
on the subjects of this article, and also to the referee for pointing out an error in 
an earlier version.}



\begin{abstract}
We address the question of whether geometric conditions on the given data can
be preserved by a solution in (1) the Whitney extension problem, and (2) the
Brenner-Fefferman-Hochster-Koll\'ar problem, both for $\cC^m$ functions. Our results
involve a certain loss of differentiability.

Problem (2) concerns the solution of a system of linear equations
$A(x) G(x) = F(x)$, where $A$ is a matrix of functions on $\IR^n$, and $F,\,G$ are
vector-valued functions. Suppose the entries of $A(x)$ are semialgebraic (or, more
generally, definable in a suitable $o$-minimal structure). Then we find $r=r(m)$ such
that, if $F(x)$ is definable and the system admits a $\cC^r$ solution $G(x)$, then
there is a $\cC^m$ definable solution. Likewise in problem (1), given a closed definable
subset $X$ of $\IR^n$, we find $r=r(m)$ such that if $g: X\to \IR$ is definable and
extends to a $\cC^r$ function on $\IR^n$, then there is a $\cC^m$ definable
extension.
\end{abstract}

\date{April 19, 2021}
\maketitle

\setcounter{tocdepth}{1}
\tableofcontents

\section{Introduction}\label{sec:intro}
This article deals with geometric versions of the following two problems on $\cC^m$ functions.

\medskip\noindent
\emph{The Whitney extension problem.} Given a closed subset
$X$ of $\IR^n$ and a function $g: X\to \IR$, how can we determine whether $g$ extends to a
$\cC^m$ function $G$ on $\IR^n$? The Whitney extension problem was solved by Charles
Fefferman, who discovered a necessary and sufficient criterion for the existence of an 
extension $G$ \cite{Fef06}, building on work of Glaeser, Bierstone--Milman--Paw{\l}ucki 
and Brudnyi--Shvartsman (see \cite{Gla58}, \cite{BMP03}, \cite{BS01}).

\medskip\noindent
\emph{The Brenner--Fefferman--Hochster--Koll\'ar problem.} 
Given a system of equations $A(x)G(x) = F(x)$, where $A(x)$ and $F(x)$ are matrix-
and vector-valued functions (respectively) on $\IR^n$ (with no \emph{a priori} assumptions
on $A$ and $F$), how can we determine whether there is a $\cC^m$ solution $G(x)$? 
This problem was formulated by Fefferman and Koll\'ar in their paper \cite{FK13} on
the \emph{Brenner--Hochster--Koll\'ar problem} (the case that
the entries of $A(x)$ and the components of $F(x)$ are polynomials, and $m=0$; see 
\cite{Bre06}, \cite{EH17}, \cite{Kol12}). A necessary and sufficient criterion for the
existence of a solution $G$ in the Brenner--Fefferman--Hochster--Koll\'ar problem was
given by Fefferman and Luli \cite{FL14}.

\medskip
In both problems above, if the given data satisfies natural geometric properties, it is reasonable to 
ask whether there is a $\cC^m$ solution satisfying the same properties; in particular, we can ask the following.

\begin{Q}\label{q:wep.semialg} Let $g: X\to \IR$ denote a semialgebraic function defined on a
closed subset $X$ of $\IR^n$ (or, more generally, a function which is definable in an $o$-minimal
structure). Assume that $g$ extends to a $\cC^m$ function on $\IR^n$. Does $g$ extend to a
semialgebraic (or definable) $\cC^m$ function?
\end{Q}

An $o$-minimal structure in this article is always understood to be an expansion of $(\IR, +,\cdot)$.
We recall that a function is called \emph{semialgebraic} (or \emph{definable}) if its graph is
semialgebraic (or definable). If $f: X\to \IR$ is definable, then $X$ is a projection of the graph,
hence also definable.

\begin{Q}\label{q:bfhk.semialg} Consider a system of equations as in the
Brenner--Fefferman--Hochster--Koll\'ar problem,
where $A(x)$ and $F(x)$ are semialgebraic (or definable in an $o$-minimal structure). Assume
there is a $\cC^m$ solution $G(x)$. Is there a semialgebraic (or definable) $\cC^m$ solution?
\end{Q}

Problem \ref{q:wep.semialg} was raised by Bierstone and Milman in \cite{Zobin}. In \cite{FK13},
Fefferman and Koll\'ar raised Problem \ref{q:bfhk.semialg}, and showed that, if the entries of 
$A(x)$ and $F(x)$ are polynomials, then there is a continuous semialgebraic solution $G(x)$.
Aschenbrenner and Thamrongthanyalak have given
positive answers for Problem \ref{q:wep.semialg} in the case $m=1$, and for Problem \ref{q:bfhk.semialg} 
in the case $m=0$, using a definable version of the Michael selection theorem \cite{AT19}.
Recently, Fefferman and Luli have obtained positive answers to both questions for arbitrary $m$,
in the case that the dimension $n=2$ (at least in the semialgebraic case \cite{FL20}).

In this article, we give solutions to both problems \emph{modulo a certain loss of differentiability}, 
for given data that is semialgebraic or, more generally, 
definable in an expansion of the real field by restricted quasianalytic functions (see Remark
\ref{rem:model});
i.e., we show that, in each case, there is a function $r: \IN \to \IN$ such that, if the problem
admits a $\cC^{r(m)}$ solution, then it admits a semialgebraic (or definable) $\cC^m$ solution.
See Theorem \ref{thm:main} and Corollary \ref{cor:wep}.  In certain cases, we can show that there is a linear function $r(m)$; see Theorem \ref{thm:linear}.

There are several other results and examples in the literature that are related to Problem \ref{q:bfhk.semialg} and the results above. In the case that the entries of $A(x)$ and $F(x)$ are \emph{regulous} (i.e., rational functions that extend continuously to $\IR^n$) rather than just semialgebraic, Koll\'ar and Nowak gave a counterexample to the existence of a regulous solution $G(x)$; they found a single linear equation
$A(x)G(x) = F(x)$ (i.e., $A(x)$ is a $1 \times q$ matrix), where $F(x)$ and the
the entries of $A(x)$ are polynomial functions on $\IR^3$, which admits a continuous semialgebraic
solution $G(x)$, but no regulous solution \cite{KN15}. On the other hand, Kucharz and Kurdyka proved that 
(in the case of a single linear equation), if $F(x)$ and the
the entries of $A(x)$ are regulous functions on $\IR^2$ and there is a continuous solution $G(x)$,
then there is a regulous solution \cite{KK17}.
Adamus and Seyedinejad gave a counterexample
to the analogue of Problem \ref{q:bfhk.semialg} with \emph{semialgebraic} replaced by
\emph{semialgebraic and arc-analytic} (i.e., analytic on every real-analytic arc) \cite{AS18}.

Let us begin by stating our main result on Problem \ref{q:bfhk.semialg}, since we will show that
it is a consequence of another assertion (Theorem \ref{thm:varphi} below), from which our
result on the extension problem \ref{q:wep.semialg} also follows (Corollary \ref{cor:wep}).

\begin{thm}\label{thm:main}
Let $A(x)$ denote a $p \times q$ matrix whose entries are functions on $\IR^n$ that are
semialgebraic (or, more generally, definable in an expansion of the reals by restricted
quasianalytic functions).
Then there is a function $r: \IN \to \IN$ with 
the following property.  For all $m\in \IN$, if $F: \IR^n \to \IR^p$ is semialgebraic (or definable) 
and the system of equations
\begin{equation}\label{eq:main}
A(x)G(x) = F(x),
\end{equation}
admits a $\cC^{r(m)}$ solution $G(x)$, then there is a semialgebraic (or definable) $\cC^m$ solution.
\end{thm}

We will deduce Theorem \ref{thm:main} from Theorem \ref{thm:varphi} following, 
concerning the solution of
a system of equations which is more complicated than \eqref{eq:main} because it involves
both matrix multiplication and composition with a mapping, but also simpler because 
the mapping and the entries of the matrix
are assumed to be both $\cC^\infty$ and semialgebraic (or definable in any polynomially
bounded $o$-minimal structure). A function that is
$\cC^\infty$ and semialgebraic is called a \emph{Nash function}; a Nash function is also 
real-analytic. 

Theorem \ref{thm:main} in the case that $A$ is $\cC^\infty$ and definable in a
polynomially bounded $o$-minimal structure holds with linear loss of differentiability; i.e.,
with a function $r(m) = \la m + \mu$, for all $m\in\IN$, where $\la,\mu\in\IN$ (see
Theorem \ref{thm:linear}).

\begin{thm}\label{thm:varphi}
Let $M$ denote a $\cC^\infty$ submanifold of $\IR^N$ (for some $N$)
which is semialgebraic (or definable in a given polynomially
bounded $o$-minimal structure). 
Let $A(x)$ denote a $p\times q$ matrix whose entries are Nash
(or $\cC^\infty$ definable) functions $A_{ij}:M\to\IR$, and let $\varphi:M\to\IR^n$ be a proper 
Nash (or $\cC^\infty$ definable) mapping. Then there is a function $r: \IN \to \IN$ with the
following property. If $f: M \to \IR^p$ is semialgebraic (or definable), and the 
system of equations
\begin{equation}\label{eq:varphi}
f(x) = A(x) g(\vp(x))
\end{equation} 
admits a $\cC^{r(m)}$ solution $g: \IR^n \to \IR^q$, then there is a semialgebraic (or definable)
$\cC^m$ solution. 
\end{thm}

\begin{rmk}\label{rmk:varphi}
We will actually prove a stronger statement than Theorem \ref{thm:varphi}:
instead of the assumptions that $f$ is definable and \eqref{eq:varphi} admits a $\cC^{r(m)}$ solution $g$, 
it is enough to assume that $f$ is $\cC^{r(m)}$ and definable, and
there is a formal solution to order $r(m)$ at each
$b \in \varphi(M)$; i.e., for each $b$, there is a (vector-valued) polynomial function $P(y)$ such that 
$f - A\cdot (P\circ\vp)$
is $r(m)$-\emph{flat} on $\vp^{-1}(b)$ (the latter means that $f - A\cdot (P\circ\vp)$ vanishes
to order $r(m)$ at every point of $\vp^{-1}(b)$). See Theorem \ref{thm:formal};
the proof follows closely that of 
\cite[Thm.\,1.2]{BMP96}. Theorem \ref{thm:formal}
is not true without some
properness assumption on the mapping $\vp$ (see Remark \ref{rmk:proper}), but Theorem \ref{thm:varphi}
may well be true with a weaker assumption than proper.
\end{rmk}

\begin{rmk}\label{rem:model}
To deduce Theorem \ref{thm:main} and Corollary \ref{cor:wep} from Theorem \ref{thm:varphi}, we will use 
the fact that, if $X$ is a closed subset of $\IR^n$ which is semialgebraic (or definable
in the given structure),
then there is a semialgebraic (or definable)
$\cC^\infty$ submanifold $M$ of $\IR^N$, 
of dimension $= \dim X$, and a proper Nash (or $\cC^\infty$ definable) mapping 
$\vp: M \to \IR^n$ such that $\vp(M) = X$. This statement (the \emph{uniformization theorem};
cf. \cite{BM88-ihes})
is a consequence of resolution of singularities in quasianalytic classes \cite{BM04}.

Given a polynomially bounded $o$-minimal structure, in general, the
class of $\cC^\infty$ functions which are locally definable is a \emph{quasianalytic class}, 
as defined axiomatically in \cite{BM04}; the quasianalyticity axiom is satisfied
according to a result
of C.~Miller \cite{Mil95}. On the other hand, given a quasianalytic class $\cC$, let $\IR_\cC$
denote the expansion of the real field by restricted functions of class $\cC$ (i.e., restrictions
to closed cubes of functions of class $\cC$, extended to $0$ outside the cube). Then $\IR_\cC$
is an $o$-minimal structure, and $\IR_\cC$ is polynomially bounded and model-complete,
and admits $\cC^\infty$ cell decomposition \cite{RSW}.

Given a quasianalytic class $\cC$, we can define sub-quasianalytic sets and 
sub-quasianalytic functions in an obvious way generalizing subanalytic. (In particular, a
sub-quasianalytic function is a function whose graph is sub-quasianalytic.) 
Resolution of singularities of an ideal generated by quasianalytic
functions holds, by \cite{BM04} (even if the ideal is not finitely generated 
\cite{BMV15}). The uniformization theorem for
sub-quasianalytic sets follows from resolution of singularities in exactly the same way
as the uniformization theorem for subanalytic sets \cite[Thm.\,0.1]{BM88-ihes} follows
from resolution of singularities in the case that $\cC$ is the class of real-analytic functions.

If $\cC$ is a quasianalytic class, then a set which is definable in $\IR_\cC$ is
sub-quasianalytic. So the uniformization theorem applies to sets that are definable
in $\IR_\cC$.

In Theorem \ref{thm:main}, there is no assumption on $A$ and $F$ other than definability
in the given structure,; we 
apply the uniformization
theorem to the closure of the graph of $A$ (likewise, to $X$ in Corollary \ref{cor:wep}). 
In Theorem \ref{thm:varphi}
(or Theorem \ref{thm:linear}),
on the other hand, $A$ and $\vp$
are $\cC^\infty$ and therefore quasiananalytic, as explained above; as a consequence, all geometric 
constructions (e.g., stratification) in the proof of the theorem are sub-quasianalytic.
See also Remark \ref{rmks:strat}(2).
\end{rmk}

We will reduce Theorem \ref{thm:main} to the assertion of Theorem \ref{thm:varphi} in the
special case that the image $\vp(M) = \IR^n$. (We will still need to use the uniformization theorem
to obtain this reduction.) But the general version of Theorem \ref{thm:varphi} is of interest because 
the composite function case ($A(x) = I$) immediately gives our result
on Problem \ref{q:wep.semialg}---we can take $\vp: M\to \IR^n$ to be a proper Nash mapping
such that $\vp(M)=X$, as given by the uniformization theorem.

\begin{cor}\label{cor:wep} 
Let $X$ denote a closed subset of $\IR^n$ which is semialgebraic (or definable in 
an expansion of the reals by restricted quasianalytic functions).
Then there is a function $r: \IN \to \IN$ such that, for all $m\in \IN$,
if $g: X \to \IR$ is a semialgebraic (or definable) function and $g$ extends to a $\cC^{r(m)}$ function
on $\IR^n$, then $g$ also has a $\cC^m$ semialgebraic (or definable) extension.
\end{cor}

We conclude this introduction by showing how Theorem \ref{thm:main} follows from
Theorem \ref{thm:varphi}. Sections \ref{sec:div}--\ref{sec:main} are devoted to the proof of Theorem
\ref{thm:varphi}, which closely follows that of the composite function theorem
in \cite{BMP96}. Theorem \ref{thm:linear} on linearity of $r(m)$ requires some of
the technology of \cite{BM17}, and we give only a sketch of the proof in Section \ref{sec:linear}.

\begin{proof}[Proof of Theorem \ref{thm:main}] Consider a system of equations \eqref{eq:main},
where $A$ and $F$ satisfy the hypotheses of Theorem \ref{thm:main}.
We regard $A(x)$ as a mapping $\IR^n \to \IR^{pq}$, and write
$a(x)=\left(1+\|A(x)\|^2\right)^{-1}$, where $\|\cdot\|$ denotes the Euclidean norm (or any norm) on $\IR^{pq}$.
Let $\Gamma$ denote the graph of the mapping $(aA,a)$ in $\IR^n \times \IR^{pq} \times \IR$. 
The closure $C$ of $\Gamma$ is a definable set of dimension $n$. By the uniformization theorem, 
there exists a proper $\cC^\infty$ definable mapping $\Phi=(\varphi, B, b):M\to\IR^n\times\IR^{pq}\times\IR$ 
from a $\cC^\infty$ definable manifold $M$ onto $C$. The
mapping $\vp: M \to \IR^n$ is also proper, because of the choice of $a(x)$ (compare with 
Remark \ref{rmk:proper}).

We claim that $\Phi$ can be taken with the property that 
$((aA)(\varphi(y)),a(\varphi(y)))=(B(y),b(y))$ on a definable open dense subset of $M$.
To see this, let $C_n$ denote the closure of the $n$-dimensional part of $\Gamma$ (the set of
points $\xi\in\Gamma$ such that $\Gamma$ has dimension $n$ in any neighbourhood of $\xi$).
Inductively, for each $k=n-1, n-2,\ldots, 0$, let $C_k$ denote the closure of the $k$-dimensional
part of $\Gamma\backslash (C_{k+1} \cup \cdots \cup C_n)$. Then each $C_k$ is a definable
set of pure dimension $k$ (or empty). By the uniformization theorem, for each nonempty $C_k$, there is a
proper $\cC^\infty$ definable mapping $\Phi_k =(\varphi_k, B_k, b_k): M_k\to\IR^n\times\IR^{pq}\times\IR$ 
from a $\cC^\infty$ definable manifold $M_k$ onto $C_k$, where $\Phi_k$ and $\vp_k$ have 
generic rank $k$ on each component of $M_k$. Then each $\Phi_k$ satisfies the required property.
Let $M$ be the disjoint union of the $M_k$, and let $\Phi$ be the mapping given by
$\Phi_k$ on each $M_k$. Then $\Phi$ has the property claimed.
 
Since $\varphi$ is surjective and $\Gamma \subset \Phi(M)$, for each $x\in\IR^n$, there exists 
$z\in M$ such that $x=\varphi(z)$, $a(x)A(x)=B(z)$ and $a(x)=b(z)$. 

We apply Theorem \ref{thm:varphi} to a system of equations
\begin{equation*}
f(y) = B(y) g(\vp(y))
\end{equation*}
on $M$; let $r: \IN \to \IN$ denote the function given by Theorem \ref{thm:varphi}.

Suppose that the equation $A(x) G(x) = F(x)$ (i.e., the equation \eqref{eq:main}) admits a
$\cC^{r(m)}$ solution $G(x)$. Then $f(y) := B(y)G(\vp(y))$ is $\cC^{r(m)}$, and, therefore, $f(y)$
is definable because there is a definable open dense subset of $M$ on which
$$
B(y)G(\vp(y)) = a(\vp(y))A(\vp(y)) G(\vp(y)) = a(\vp(y)) F(\vp(y)).
$$

By Theorem \ref{thm:varphi}, there exists a definable $\cC^m$ function $\tG: \IR^n \to \IR^q$
such that $B(y)G(\vp(y)) = f(y) = B(y)\tG(\vp(y))$, for all $y\in M$.

Now, for all $x\in \IR^n$, there exists $z\in M$ such that $x = \vp(z)$,
$(aA)(x) = (aA)(\vp(z)) = B(z)$, and $a(x) = a(\vp(z)) = b(z)$; therefore,
\begin{align*}
a(x)F(x) &= a(\vp(z))F(\vp(z))\\
                            &= a(\vp(z))A(\vp(z))G(\vp(z))\\
                            &= B(z) G(\vp(z))\\
                            &= B(z) \tG(\vp(z)) = a(x) A(x) \tG(x),
\end{align*}
so that $F(x) = A(x) \tG(x)$, as required.
\end{proof}

\section{Hironaka's division algorithm}\label{sec:div}

\begin{notation}
We use standard multiindex notation: Let $\IN$ denote the nonnegative integers. 
If $\al = (\al_1,\ldots,\al_n) \in \IN^n$ and $x=(x_1,\ldots,x_n)$, we write $|\al| := \al_1 +\cdots +\al_n$, 
$\al! := \al_1!\cdots\al_n!$, $x^\al := x_1^{\al_1}\cdots x_n^{\al_n}$,
and $\p^{|\al|} / \p x^{\al} := \p^{\al_1 +\cdots +\al_n} / \p x_1^{\al_1}\cdots \p x_n^{\al_n}$. 

Let $A$ be a ring. Let $A\llb x\rrb = A\llb x_1,\ldots,x_n\rrb$
denote the ring of formal power series in $n$ indeterminates with coefficients in $A$. Given 
$F = F(x) \in A\llb x\rrb^p$,
we write $F(x) = \sum_{j=1}^p \sum_{\al\in\IN^n} F_{(\al,j)}x^{(\al,j)}$, where each coefficient
$F_{(\al,j)}\in A$, and $x^{(\al,j)}:= (0,\ldots,x^\al,\ldots,0)$ with $x^\al$ in the $j$th place. If $(\al,j) \in
\IN^n\times\{1,\ldots,p\}$ and $\be \in \IN^n$, we define $(\al,j) + \be := (\al +\be,j)$. We will use the
following notation:
\begin{align*}
\supp F &:= \{(\al,j)\in \IN^n\times\{1,\ldots,p\}:\, F_{(\al,j)}\neq 0\},\\
\exp F &:= \min\,\supp F,\\
\mon F &:= F_{\exp F}x^{\exp F},
\end{align*}
where $\min$ is with respect to the total ordering of $\IN^n\times\{1,\ldots,p\}$ given by 
$\lex (|\al|,j,\allowbreak\al) = \lex (|\al|,j,\allowbreak\al_1,\ldots,\al_n)$, 
where $\lex$ denotes lexicographic order. We call $\exp F$
the \emph{initial exponent}, $\mon F$ the \emph{initial monomial} and $F_{\exp F}$ the
\emph{initial coefficient} of $F$. 

Let $\fm$ or $(x) = (x_1,\ldots,x_n)$ denote the maximal ideal of $\IR\llb x\rrb$; $(x)$ is the
ideal generated by $x_1,\ldots,x_n$.
\end{notation}

\begin{rmks}\label{rem:order}
(1) In the following,
instead of the preceding order, we may also use any total ordering of $\IN^n\times\{1,\ldots,p\}$ which is
compatible with addition in the sense that $(\al,j) + \be > (\al,j)$ if $\be \in \IN^n\setminus\{0\}$.
For example, we may use $\lex (L(\al),j,\al)$, where $L$ is a positive linear form 
$L(\al) = \sum_{i=1}^n \la_i\al_i$ (i.e., all $\la_i$ are positive real numbers).

\medskip\noindent
(2) For the proof of our main theorem, we will need the results of this section only in the case
that the coefficient ring $A$ is the field $\IR$. The division results here are nevertheless stated with a 
general coefficient ring because there is essentially no extra cost involved, and because the proof of
Theorem \ref{thm:linear} (which is only sketched in Section \ref{sec:linear}) depends on the 
division algorithm with more general coefficients.
\end{rmks}

\subsection{Hironaka's formal division algorithm \cite{Hir64}, \cite[\S6]{BM87},
\cite[\S2]{BM17}} \label{subsec:formaldiv}
\begin{thm}\label{thm:formaldiv}
Let $L$ denote a positive linear form on $\IR^n$, and consider the corresponding ordering of
$\IN^n\times\{1,\ldots,p\}$ (as in Remark \ref{rem:order}(1)).
Let $\Phi_1,\ldots,\Phi_q\in A\llb x\rrb^p$. Set $(\al_i,j_i):=\exp \Phi_i$, $i=1,\ldots,q$, and consider the
partition $\{\De_1,\ldots\De_q,\De\}$ of $\IN^n\times\{1,\ldots,p\}$ given by $\De_1 := (\al_1,j_1) +\IN^n$,
\begin{align*}
\De_i &:= (\al_i,j_i) +\IN^n \setminus \bigcup_{k=1}^{i-1}\De_k,\quad i=2,\ldots,q,\\
\De &:= \IN^n\times\{1,\ldots,p\} \setminus \bigcup_{i=1}^{q}\De_i.
\end{align*}
Suppose that $A$ is an integral domain, and let $\IK$ be the field of fractions of $A$.
Let $S$ denote the multiplicative subset of $A$ generated by the initial coefficients $\Phi_{i,(\al_i,j_i)}$,
and let $B$ denote any subring of $\IK$ containing the localization $S^{-1}A$.
Then, for every $F \in B\llb x\rrb^p$, there exist unique $Q_i = Q_i(F) \in B\llb x\rrb$, $i=1,\ldots,q$, and
$R = R(F) \in B\llb x\rrb^p$ such that
\begin{align*}
F &= \sum_{i=1}^q Q_i\Phi_i + R,\\
(\al_i,j_i) + \supp Q_i \subset &\De_i,\,\,\, i=1,\ldots,q,\,\,\, \text{and }\,\, \supp R \subset \De.
\end{align*}
Moreover,
$$
(\al_i,j_i) +\exp Q_i \geq \exp F,\,\,\, i=1,\ldots,q,\,\,\, \text{and}\,\, \exp R \geq \exp F.
$$
\end{thm}

See \cite[\S2]{BM17} for a simple proof.

\subsection{The diagram of initial exponents \cite[\S2]{BM17}}\label{subsec:diag} 
Let $A$ be a ring, and let $M$ be a submodule
of $A\llb x\rrb^p$. We define the \emph{diagram of initial exponents} $\cN(M) \subset \IN^n\times \{1,\ldots,p\}$ as
$$
\cN(M) := \{\exp F:\, F \in M\setminus \{0\}\}.
$$
Clearly, $\cN(M) + \IN^n = \cN(M)$.
We say that $(\al_0,j_0)\in \IN^n\times \{1,\ldots,p\}$ is a \emph{vertex} of $\cN(M)$ if 
$(\cN(M)\setminus\{(\al_0,j_0)\}) + \IN^n \neq \cN(M)$. It is easy to see that $\cN(M)$ has finitely many vertices.
Let $(\al_i,j_i)$, $i=1,\ldots,q$, denote the vertices of $\cN(M)$. Choose $\Phi_i \in M$ such that
$(\al_i,j_i) = \exp \Phi_i$, $i=1,\ldots,q$. (We call $\Phi_i$ a \emph{representative} of the vertex $(\al_i,j_i)$.)

\begin{cor}\label{cor:diag}
Assume that $A$ is an integral domain. Then,
using the notation of Theorem \ref{thm:formaldiv}, we have:
\begin{enumerate}
\item\begin{enumerate}
\item $\cN(M) = \bigcup_{i=1}^q \De_i$;
\smallskip\item $\Phi_1,\ldots,\Phi_q$ generate $B\llb x\rrb\cdot M$;
\smallskip\item if $G \in B\llb x\rrb^p$, then $G \in B\llb x\rrb\cdot M$ if and only if $R(G) = 0$.
\end{enumerate}
\medskip\item There exist unique generators $\Psi_i$, $i=1,\ldots,q$, of $S^{-1}A\llb x\rrb\cdot M$
such that
$$
\Psi_i = x^{(\al_i,j_i)} + R_i,\quad \text{where }\, \supp R_i \subset \De.
$$
\end{enumerate}
\end{cor}

\begin{proof}
(1) (a) is obvious. Let $G \in B\llb x\rrb^p$. Write $G = \sum_{i=1}^q Q_i(G) \Phi_i + R(G)$, according to the
formal division algorithm. Then $G \in B\llb x\rrb\cdot M$ if and only if $R(G) \in M$; i.e., if and 
only if $R(G)=0$ (since $\supp R(G) \bigcap \cN(M) = \emptyset$). (b), (c) follow immediately.

\medskip\noindent
(2) For each $i$, write $x^{(\al_i,j_i)} = \sum_{j=1}^q Q_{ij} \Phi_j - R_i$, according to the division
algorithm. Clearly, we can take $\Psi_i = x^{(\al_i,j_i)} +R_i$, $i=1,\ldots,q$. 
\end{proof}

We call $\Psi_1,\ldots,\Psi_q$ in Corollary \ref{cor:diag} the \emph{standard basis} of $M$ (or of
$S^{-1}A\llb x\rrb\cdot M$). The following is a consequence of Theorem \ref{thm:formaldiv} in
the case $A=\IR$.

\begin{cor}\label{cor:compl}
Let $M$ be a submodule of $\IR\llb x \rrb^p$ and let $r\in\IN$. Then the
set $\left\{x^{(\beta,i)}: (\beta,i)\in\Delta,\,|\beta|\le r\right\}$ is a basis of an $\IR$-linear complement of 
$M+(x)^{r+1}\IR\llb x\rrb^p$ in $\IR\llb x\rrb^p$ (where $\Delta$ is the complement of $\cN(M)$).
\end{cor}

\begin{cor}\label{cor:AR}
Let $A_1,\ldots A_q \in \IR\llb x\rrb^p$. Let $M\subset \IR\llb x\rrb^p$ denote the module generated
by $A_1,\ldots A_q$; i.e., $M = \Im A$, where $A: \IR\llb x \rrb^q \to \IR\llb x \rrb^p$ is the module
homomorphism given by multiplication by the matrix $A$ with columns $A_1,\ldots A_q$. Let 
$R := \Ker A \subset \IR\llb x\rrb^q$ (the \emph{submodule of relations among the columns of $A$}).
Let $(\al_i,j_i)$, $i=1,\ldots,t$, denote the vertices of $\cN(M)$, and set $\la := \max |\al_i|$. Then:
\begin{enumerate}
\item[(1)] (Artin-Rees Lemma)\quad for all $l \in \IN$,\  \ $M \cap \fm^{l+\la}\cdot \IR\llb x\rrb^p 
= \fm^l\cdot(M \cap \fm^{\la}\cdot \IR\llb x\rrb^p)$;
\item[(2)] (Chevalley estimate)\quad $A^{-1}(\fm^{l+\la}\cdot \IR\llb x\rrb^p) \subset R + \fm^l\cdot\IR\llb x\rrb^q$.
\end{enumerate}
\end{cor}

\begin{proof}
(1) As above, let $\Phi_i$ denote representatives of the vertices $(\al_i,j_i)$.
By Theorem \ref{thm:formaldiv}, if $F\in M$, then $F = \sum_{i=1}^t Q_i \Phi_i$,
where $(\al_i,j_i) + \supp Q_i \subset \De_i$, for each $i$, and, if $F \in \fm^l\cdot\IR\llb x\rrb^p$,
then each $Q_i \subset \fm^{l-|\al_i|}$. Since each $\Phi_i \subset \fm^{|\al_i|}\cdot\IR\llb x\rrb^p$,
we get $M \cap \fm^{l+\la}\cdot \IR\llb x\rrb^p \subset \fm^l\cdot(M \cap \fm^{\la}\cdot \IR\llb x\rrb^p)$.
The reverse inclusion is obvious.

\medskip\noindent
(2) It is easy to see that, for all $l,\la \in\IN$, the Chevalley estimate (2) is equivalent to the condition,
$$
M \cap \fm^{l+\la}\cdot \IR\llb x\rrb^p \subset \fm^l \cdot M,
$$
so the assertion follows from (1).
\end{proof}

We totally order the set of diagrams
\begin{equation}\label{eq:diag}
\cD(n,p):= \{\cN \subset \IN^n\times\{1,\ldots,p\}:\, \cN + \IN^n = \cN\},
\end{equation}
as follows. Given $\cN \in \cD(n,p)$, let $\cV(\cN)$ denote the sequence 
obtained by listing the vertices of $\cN$ in increasing order and completing
this list to an infinite sequence by using $\infty$ for all the remaining terms. If $\cN_1, \cN_2 \in \cD(n,p)$
we say that $\cN_1<\cN_2$ if $\cV(\cN_1)< \cV(\cN_2)$ with respect to the lexicographic ordering on the
set of such sequences. Clearly, $\cN_1\leq\cN_2$ if $\cN_1 \supseteq \cN_2$.

\begin{rmk}\label{rem:diagdiff}
The condition $\cN + \IN^n = \cN$ implies that the set of elements of $\IR\llb x\rrb^p$ with
supports in the complement of $\cN$ is closed under formal differentiation. This simple property
of the diagram will play an important part in the proof of our main theorem (see the proof of
Lemma \ref{lem:GtauWF}).
\end{rmk}

\section{Whitney fields}

\begin{notation}\label{notation}
Let $a \in \IR^n$ or, more generally, let $a\in M$, where $M$ is a $\cC^\infty$ manifold of dimension $n$.
Let $x=(x_1,\ldots,x_n)$ denote the affine coordinates of $\IR^n$, or local coordinates
for $M$ in a neighbourhood of $a$. The Taylor expansion $T^m_a f$ 
of order $m\in\IN$ at $a$ of a $\cC^m$
function $f(x_1,\ldots,x_n)$ is written 
$$
T^m_af(x) 
= \sum_{\substack{\al\in\IN^n\\|\al|\leq m}}\frac{1}{\al!}\left(\frac{\p^{|\al|}f}{\p x^\al}\right)(a)\, x^\al\,.
$$
Likewise, if $m = \infty$, $T_af(x) = T^\infty_af(x) = \sum_{\al\in\IN^n}(\p^{|\al|}f/\p x^\al)(a)x^\al/\al!$.
In particular, the ring of formal power series centred at $a$ is identified with the ring of power series 
$\IR\llb x\rrb = \IR\llb x_1,\ldots,x_n\rrb$.

Given $m\in \IN$ or $m=\infty$, we write $\tT^m_af(x) = T^m_af(x) - f(a)$ (the Taylor expansion
without constant term).
\end{notation}

\begin{defi}\label{def:wh}
A \emph{$\cC^m$-Whitney field} $F$ on a locally closed subset $E$ of $\IR^n$ is a polynomial
$$
F(a,x)=\sum_{|\alpha|\le m} \frac{1}{\alpha!}f_\alpha(a)x^\alpha\,\in\,\cC^0(E)[x]
$$
with continuous coefficients on $E$, such that, for all $\al\in\IN^n$, $|\al|\leq m$, 
$$
\frac{\p^{|\al|}F}{\p x^{\al}}(a,0) - \frac{\p^{|\al|}F}{\p x^{\al}}(b,a-b)=o\left(|a-b|^{m-|\al|}\right)
$$
on $E$; i.e., the quotient by $|a-b|^{m-|\al|}$ of the left-hand side $\to 0$ as $a,b\in E$ 
tend to a point $c\in E$.

Given an $o$-minimal structure, we say that $F$ is a \emph{definable Whitney field}
if the coefficients $f_\al$ are definable.
\end{defi}

We will use the following definable version of the Whitney extension theorem \cite{Whi34}, \cite{Mal67} 
due to Kurdyka and Paw\l{}ucki \cite{KP14} and Thamrongthanyalak \cite{Tha17}.

\begin{thm}\label{thm:ominWhitney}
Let $m, t\in \IN$, $m\leq t$. If $F$ is a definable $\cC^m$-Whitney field on a closed subset $E\subset\IR^n$
(as above), there exists a definable $\cC^m$ function $f:\IR^n\to\IR$ such that $(\p^{|\al|}f/\p x^\al)=f_\al$ on $E$, for all $|\al|\leq m$ (in particular, $f|_E=f_0$) and $f$ is $\cC^t$ on $\IR^n\backslash E$.
\end{thm}

Given $F(a,x) = \sum_{|\alpha| \le m} f_\alpha(a)x^\alpha/\al! \in \cC^0(E)[x]$ and $k\leq m$, we write
$F^k(a,x) = \sum_{|\alpha| \le k} f_\alpha(a)x^\alpha/\al! $. The following result can be understood as a converse of the fact that Taylor expansion commutes with differentiation. 

\begin{lemma}[Borel's lemma with parameter {\cite[\S10]{BM87-2}}]\label{lem:borel}
Let $\La$ denote a $\cC^m$-submanifold of $\IR^n$ and let
$F \in \cC^0(\La)[x]$, $\deg F \leq m$. Then $F$ is a $\cC^m$ Whitney field on $\La$ if and only if
$F^{m-1} \in \cC^1(\La)[x]$ and, for all $a\in \La$ and $u\in T_a\La$ (where $T_a\La$ denotes
the tangent space of $\La$ at $a$),
$$
D_{a,u}F^{m-1}(a,x) = D_{x,u} F(a,x),
$$
where $D_{a,u}$ denotes the directional derivative with respect to $a$ in the direction $u$,
and $D_{x,u}$ denotes the (formal) directional derivative with respect to $x$ in the direction $u$.
\end{lemma}

\begin{proof}
It is enough to assume that $\Lambda=\IR^k\times\{0\}$, and the result is then straightforward.
\end{proof}

Let $\rho$ be a positive integer. We say that a closed subset $A$ of $\IR^n$ is \emph{$\rho$-regular}
if any point of $A$ admits a neighbourhood $U$ in $\IR^n$ with the property that there exists 
$C>0$ such that any two points $a,b\in A\cap U$ can be joined by a rectifiable curve in $A$
of length $\leq C |a-b|^{1/\rho}$. 

\begin{lemma}\label{lem:rreg}
Given an expansion of the real field by restricted quasianalytic functions,
any definable closed subset of $\IR^n$ is $\rho$-regular, for some $\rho$.
\end{lemma}

This follows from resolution of singularities (more precisely, from the uniformization theorem)
and {\L}ojasiewicz's inequalities, following the argument given in \cite[Thm.\,6.10]{BM88-ihes}
for the subanalytic case.

\begin{lemma}[l'H\^opital--Hestenes lemma {\cite[Prop.\,3.4]{BMP96}}]\label{lem:hestenes} 
Let $B\subset A$ denote closed subsets of $\IR^n$, where $A$ is $\rho$-regular. Let $m\in\IN$
and let $F(a,x) = \sum_{|\al|\leq m\rho} f_\al(a)x^\al/\al! \in \cC^0(A)[x]$. Assume that 
$F$ restricts to $\cC^{m\rho}$ Whitney fields on $B$ and on $A\backslash B$. Then
$F^m(a,x) = \sum_{|\al|\leq m} f_\al(a)x^\al/\al!$ is a $\cC^m$ Whitney field on $A$.
\end{lemma}

\section{Stratification of a definable mapping and fibre product}\label{sec:strat}
Theorem \ref{thm:strat} below, concerning stratification of a definable mapping, 
is a simple consequence of $\cC^\infty$ cell decomposition. There is a stronger
\emph{local trivialization theorem} (cf. \cite{Now11}), which generalizes results of
Hardt for the semialgebraic or subanalytic case \cite{Har76,Har80}, but the
stronger assertion will not be needed here.

Let $E$ denote a subset of $\IR^n$ which is definable in a given $o$-minimal structure.
A \emph{definable stratification} of $E$ is a partition $\Ss$ of $E$ 
(locally finite in $\IR^n$) into connected definable $\cC^\infty$ submanifolds,
called \emph{strata}, such that, for each
$S\in \Ss$, the boundary $(\overline{S}\backslash S) \cap E$ is a union of strata of dimension
$< \dim S$. A partition $\P$ of $E$ is \emph{compatible} with a family $\Fs$ of subsets of $E$
if, for each $P\in \P$ and $F\in\Fs$, either $P\subset F$ or $P\subset E\backslash F$. 

\begin{thm}\label{thm:strat}
Let $\cR$ denote an $o$-minimal structure which admits $\cC^\infty$ cell decomposition.
Let $E$ denote a definable subset of $\IR^N$, and $\vp: E\to \IR^n$ a continuous definable
mapping. Put $X=\vp(E)$. Let $\Fs$ and $\Gs$ denote finite families of definable subsets of $E$ and
$X$, respectively. Then there are finite definable stratifications $\Ss$ and $\Ts$ of $E$ and $X$,
respectively, such that
\begin{enumerate}
\item[(1)]For each $S\in \Ss$, $\vp(S) \in \Ts$ and there is a commutative diagram
 $$
 \xymatrix{S \ar[rr]^h \ar[dr]_{\varphi|_{S}} & & T\times P \ar[ld]^{\pi} \\ & T &}
 $$
 where $T=\vp(S)$, $P$ is a connected definable $\cC^\infty$ submanifold of $\IR^s$, for some $s$, $h$ is
 a $\cC^\infty$ definable isomorphism, and $\pi$ denotes the projection;
 \item[(2)] $\Ss$ is compatible with $\Fs$ and $\Ts$ is compatible with $\Gs$.
 \end{enumerate}
\end{thm}

A pair $(\Ss,\Ts)$ of definable stratifications as in (1) will be called a \emph{definable stratification} of $\vp$.
If we weaken this definition by allowing $\Ss$ to be any finite partition into connected definable $\cC^\infty$
submanifolds (but still requiring $\Ts$ to be a stratification), then the pair $(\Ss,\Ts)$ will be called a
\emph{(definable) semistratification} of $\vp$.

\begin{proof}[Proof of Theorem \ref{thm:strat}]
Let $E' \subset \IR^n\times\IR^N$ denote the image of $E$ by the mapping $x\mapsto (\vp(x),x)$
(the graph of $\vp$), and let $\Fs '$ denote the family of definable subsets of $E'$ corresponding
to $\Fs$. Let $\vp': E' \to \IR^n$ and $\pi': E' \to \IR^N$ denote the definable mappings induced by
the projections $\IR^{n+N} \to \IR^n$ onto the first $n$ coordinates, and $\IR^{n+N} \to \IR^N$ onto the 
last $N$ coordinates, respectively. There is a finite $\cC^\infty$ cell decomposition $\Cs'$
of $E'$ compatible with $\Fs'$, such that $\pi'$ has constant rank $= \dim C$ on each cell $C \in \Cs'$.

The projection $\vp'(C)$ of each cell $C \in \Cs'$ to $\IR^n$ is a cell in $\IR^n$. Let $\Ts$ denote a 
finite definable stratification (or $\cC^\infty$ cell decomposition) of $X$ compatible with $\Gs$ and
all projected cells. We can define a stratification $\Ss'$ of $E'$ so that $\Ss',\,\Ts$ satisfy the
conclusion of the theorem for $\vp'$, by taking the inverse images $S'$ in all cells of $\Cs'$, of
the strata of $\Ts$. Let $\Ss$ denote the corresponding stratification of $E$. Then $\Ss,\,\Ts$ satisfy
the conclusion of the theorem for $\vp$.
\end{proof}

\begin{rmks}\label{rmks:strat}
(1) If $(\Ss,\Ts)$ is a definable semistratification of $\vp$, and $Y$ is a definable subset of $X$
such that $\Ts' = \{T\cap Y: T\in\Ts\}$ is a stratification of $Y$, then the pair $(\Ss',\Ts')$, where
$\Ss' = \{S\cap \vp^{-1}(Y): S\in\Ss\}$, is a semistratification of $\vp|_{\vp^{-1}(Y)}: \vp^{-1}(Y)\to Y$.

\medskip\noindent
(2) Although Theorem \ref{thm:varphi} does not have $\cC^\infty$ cell decomposition as a
hypothesis, Theorem \ref{thm:strat} can be used in the proof, as explained in Remark
\ref{rem:model}. It seems worth noting that any $o$-minimal structure admits $\cC^r$ 
cell decomposition, and satisfies the $\cC^r$ version of Theorem \ref{thm:strat}.
\end{rmks}

\begin{defrmk}\label{def:fibre}
Let $\vp: E\to\IR^n$ be a continuous definable mapping, where $E$ is a definable subset
of $\IR^N$. For each $s\in\IN$, $s\geq 1$, we define the
\emph{$s$-fold fibre product of $E$ with respect to $\varphi$} as 
$$
E_\varphi^s \coloneq \left\{(a_1,\ldots,a_s)\in E^s\,: \varphi(a_1)=\cdots=\varphi(a_s)\right\}.
$$ 
Clearly, $E_\varphi^s$ is a definable subset of $(\IR^N)^s$, and there is a natural mapping
$\Phi: E_\varphi^s \to \IR^n$ defined by $\Phi(a_1,\ldots,a_s) = \vp(a_1)$. Suppose that
$(\Ss,\Ts)$ is a semistratification of $\vp$. Let $\Ss^{(s)}$ denote the family of all
nonempty sets of the form $\left(S_1\times\cdots\times S_s\right)\cap E_\varphi^s$, where
$S_1,\ldots,S_s \in \Ss$. It is easy to see (using Remarks \ref{rmks:strat}(1)) that
$(\Ss^{(s)},\Ts)$ is a semistratification of $\Phi$.
\end{defrmk}

\begin{lemma}\label{lem:finite}
Let $M$ denote a $\cC^\infty$ submanifold of $\IR^n$ which is definable in a given $o$-minimal
structure. Then $M$ can be covered by finitely many definable $\cC^\infty$ coordinate charts.
\end{lemma}

\begin{proof}
We can assume that $M$ is of pure dimension.
Let $U$ denote the open subset of $M$ consisting of all points where the projection
$\pi: M \subset \IR^n \to \IR^{\dim M}$ has rank $= \dim M$. Then $U$ is definable and $\pi|_U$
is locally a homeomorphism. Moreover, $M$ is covered by finitely many such $U$, corresponding to 
the projections $\pi$ after permutation of the coordinates of $\IR^n$.

It is enough to prove, therefore, that, given a definable mapping $\vp: U \to \IR^k$ which is locally
a homeomorphism, we can cover $U$ by finitely many definable open subsets $U_i$ such that $\vp|_{U_i}$
is a homeomorphism onto its image. 

To see this, set $V := \vp(U)$ and consider a definable
triangulation $\{\s\}$ of the closure $\oV$, compatible with $\oV\backslash V$. For each simplex $\s$,
let $\s^o$ denote the corresponding open simplex (i.e., $\s^o = \s \backslash \bdry \s$). By the
stratification theorem, we can assume that, for each $\s^o \subset V$, there is a definable homeomorphism
$\theta_\s: \vp^{-1}(\s^o) \to \s^o \times F_\s$ commuting with the projections to $\s^o$, where
$F_\s$ is a finite set (by $o$-minimality). Let $a \in F_\s$.
If $\s$ is a face of another simplex $\tau$, then there is a unique $b = b_\tau(a)\in F_\tau$
such that $\theta_\s^{-1}\left(\s^o \times \{a\}\right)$ is the intersection of $\vp^{-1} (\s^o)$ with
the closure of $\theta_\tau^{-1} \left(\tau^o\times \{b_\tau(a)\}\right)$ (since $\vp$ is locally a
homeomorphism). Let $U_{\s,a}$ denote the union of all 
$\theta_\tau^{-1} \left(\tau^o\times \{b_\tau(a)\}\right)$, where $\s$ is a face of $\tau$.
Then $\theta_\s^{-1}\left(\s^o \times \{a\}\right) \,\cup \,U_{\s,a}$ is a definable open neighbourhood of 
$\theta_\s^{-1}\left(\s^o \times \{a\}\right)$ on which $\vp$ is a homeomorphism.
\end{proof}

\section{Proof of the main theorem}\label{sec:main}
Theorem \ref{thm:varphi} is an immediate consequence of the following 
assertion, which will be proved in \S\ref{subsec:ind}, following a preparatory
subsection \S\ref{subsec:main1}.

\begin{thm}\label{thm:formal}
Let $M$ denote a $\cC^\infty$ submanifold of $\IR^N$ (for some $N$)
which is definable in a given polynomially
bounded $o$-minimal structure. 
Let $A(x)$ denote a $p\times q$ matrix whose entries are 
$\cC^\infty$ definable functions $A_{ij}:M\to\IR$, and let $\varphi:M\to\IR^n$ be a proper 
$\cC^\infty$ definable mapping. Then there is a function $r: \IN \to \IN$ with the
following property. If $f: M \to \IR^p$ is $\cC^{r(m)}$ and definable, and the 
system of equations
\begin{equation}\label{eq:varphi2}
f(x) = A(x) g(\vp(x))
\end{equation} 
can be solved formally to order $r(m)$ at each
$b \in \varphi(M)$ (i.e., for each $b$, there exists a (vector-valued) 
polynomial function $P(y)$ of degree $r(m)$
such that 
$f - A\cdot (P\circ\vp)$ vanishes
to order $r(m)$ at every point of $\vp^{-1}(b)$),
then there is a definable $\cC^m$ solution $g$. 
\end{thm}

\begin{rmk}\label{rmk:proper}
Theorem \ref{thm:formal} is not true without some properness assumption on
the mapping $\vp$. The following is an example
involving composition only (i.e., $A(x) = I$). Let $M$ denote the Nash submanifold of
$\IR^2$ given by the union of the curves $y = 1/x$, where $x > 0$, and
$x = -y^2$. Let $\vp: M \to \IR$ be the mapping given by
projection onto the $x$-axis. Then $\vp(M) = \IR$ is closed, but $\vp$ is not
proper. Any $g: \IR \to \IR$ given by the restriction of a $\cC^m$ function
to $(-\infty, 0]$ and any other $\cC^m$ function on $(0, \infty)$ pulls back
to a $\cC^m$ function $f(x,y) = g(\vp(x,y))$ on $M$, but of course $g$ need not be
even continuous. 

The hypothesis that $\vp$ be proper is needed in Proposition \ref{prop:ind} below.

In this context, we note that the composite function theorem of Nowak \cite{Now11} is not correct 
as stated;  \emph{semiproper} is a necessary hypothesis. We recall that, if $\vp: X \to Y$ is a continuous
mapping of Hausdorff spaces, then every continuous function $f:X\to\IR$ which is constant on
the fibres of $\vp$ is a composite $f = g\circ\vp$, where $g$ is continuous, if and only if
$\vp$ is semiproper \cite{BS83}.
\end{rmk}

\subsection{Module of relations}\label{subsec:main1}
The hypotheses and notation of Theorem \ref{thm:formal} will be fixed throughout the
rest of the section. By Lemma \ref{lem:finite}, $M$ admits
a finite covering $\Us$ by $\cC^\infty$ definable coordinate charts $U$ ($M$ may
have components, and therefore coordinate charts, of different dimensions.)

Let $X=\vp(M)\subset \IR^n$. Let $y=(y_1,\ldots,y_n)$ denote the affine coordinates of $\IR^n$.

\begin{defi}
The \emph{module of relations $\cR_r(b)$ of order $r$ at $b\in X$} is defined as 
$$
\R_r(b):=\left\{W\in\IR\llbracket y\rrbracket^q:  \text{for all } a\in\varphi^{-1}(b),\ \ 
T^r_aA(x)\,W(\tT^r_a\varphi(x))\equiv0\ \text{ mod }(x)^{r+1}\IR\llbracket x\rrbracket^p
\right\},
$$ 
where, for each $a\in \vp^{-1}(b)$, $x = (x_1,\ldots,x_{n'})$ denotes the coordinates of a 
chart $U\in \Us$ containing $a$, and $(x)$ denotes the ideal generated by $x_1,\ldots,x_{n'}$.
\end{defi}

Note that this definition is independent of the choice of local coordinates at each $a\in \vp^{-1}(b)$,
since it can be expressed in terms of a module homomorphism over the ring homomorphism
$\hvp^*_a$ induced by $\vp$ between the 
completed local rings of germs of $\cC^\infty$ functions at the points $b$ and $a$, for each
$a\in \vp^{-1}(b)$.

Recall the notation $W(y) = \sum_{(\be,j)} W_{\be,j} y^{\be,j} \in \IR\llb y\rrb^q$, where
$(\be,j)\in \IN^n \times \{1,\ldots,q\}$ (Section \ref{sec:div}).

Clearly, $W\in\R_r(b)$ if and only if its coefficients $W_{\be,j},\, |\be|\leq r$
satisfy a system of linear equations
\begin{equation}\label{eqn:system}
\sum_{\substack{(\beta,j)\\|\beta|\le r}}L_{\alpha,i}^{\beta,j}(a)\,W_{\beta,j}=0,\quad
\text{where }\ \al\in \IN^{n'}, |\al|\leq r,\,i=1,\ldots,p,\, a\in\varphi^{-1}(b)
\end{equation}
(and where we have chosen
a coordinate chart $U$ containing $a$ arbitrarily, for each $a$). Namely, for each
$\al, i$ and $a$, the left-hand side of \eqref{eqn:system} is the $(\al,i)$-coefficient of
the expansion $T^r_aA(x)\,W(\tT^r_a\varphi(x))$.

Set
$$
s := {n+r\choose r}q = \#\{(\be, j): |\be|\leq r\},
$$
and let $\Phi:M_\varphi^s\to X\subset\IR^n$ denote the mapping from the fibre product, as
in Definition \ref{def:fibre}.

Let $(\Ss,\Ts)$ denote a definable stratification of $\vp: M\to \IR^n$, such that $\Ss$ is
compatible with the covering $\Us$ of $M$. Then $(\Ss,\Ts)$ induces a definable
semistratification $(\Ss^{(s)},\Ts)$ of $\Phi$. Since the stratification $\Ss$ is compatible 
with the covering $\Us$ of $M$, each stratum of $\Ss$
lies in some product coordinate chart of $M^s$
(i.e., product of charts $U\in \Us$)

Given $l \in\IN$, $l\leq r$, let $\pi_l: \IR\llb y\rrb^q \to \IR[y]^q$ denote the projection
$\pi_l\left(\sum_{(\be,j)}W_{\be,j} y^{\be,j}\right)=\sum_{(\be,j),|\be|\leq l}W_{\be,j} y^{\be,j}$.
If $l\leq r\leq r'$ and $b\in X$, then $\pi_l(\cR_{r'}(b)) \subset \pi_l(\cR_{r}(b)) \subset \cR_l(b)$;
therefore, there exists $r\geq l$ such that $\pi_l(\cR_{r'}(b)) = \pi_l(\cR_r(b))$, for all $r'\geq r$.
Let $r(b,l)$ denote the smallest such $r$.

Given $\ua = (a_1,\ldots,a_s) \in M^s_\vp$\,, let $\rho^0(\ua)$ denote the rank of the matrix
of the system of equations
\begin{equation}\label{eqn:finitesystem}
\sum_{\substack{|\beta|\le r\\j=1,\ldots,q}}L_{\alpha,i}^{\beta,j}(a_\nu)\,W_{\beta,j}=0,\quad
\text{where }\, |\al|\leq r,\, i=1,\ldots,p,\, \nu=1,\ldots,s;
\end{equation}
i.e, $\rho^0(\ua)$ is the rank of the matrix $\left(L_{\alpha,i}^{\beta,j}(a_\nu)\right)$
with rows indexed by $\al, i, \nu$ and columns indexed by $\be, j$. Likewise, given
$l\leq r$, let $\rho^1(\ua)$ denote the rank of the matrix
of the system of equations
$$\sum_{\substack{l<|\beta|\le r\\j=1,\ldots,q}}L_{\alpha,i}^{\beta,j}(a_\nu)\,W_{\beta,j}=0,\quad
\text{where }\, |\al|\leq r,\, i=1,\ldots,p,\, \nu=1,\ldots,s.$$
Of course, $\rho^0(\ua)$ depends on $r$, and $\rho^1(\ua)$ depends on $r,l$.
Again, however, neither depends on the choice of coordinate chart $U$ containing each
$a_\nu$.

Let $l \in\IN$, $l\leq r$, and consider a stratum $T\in\Ts$. Set
\begin{align*}
\s^0_T &:= \max\left\{\rho^0(\ua): \ua\in\Phi^{-1}(T)\right\},\\
\s^1_T &:= \max\left\{\rho^1(\ua): \ua\in\Phi^{-1}(T)\right\}.
\end{align*}
Note that, if $b\in T$ is a point such that
\begin{align*}
\max\left\{\rho^0(\ua): \ua\in\Phi^{-1}(b)\right\} &= \s^0_T,\\
\max\left\{\rho^1(\ua): \ua\in\Phi^{-1}(b)\right\} &= \s^1_T,
\end{align*}
then 
\begin{align*}
\dim \pi_r(\cR_r(b)) &= s - \s^0_T,\\
\dim \pi_l(\cR_r(b)) &= s - \s^0_T - \left(s - {n+l \choose l} q - \s^1_T\right) = {n+l \choose l} q + \s^1_T - \s^0_T.
\end{align*}

There are strata $S^0_T,\,S^1_T \in \Ss^{(s)}$ over $T$ such that 
\begin{align*}
\s^0_T &= \max\left\{\rho^0(\ua): \ua\in S^0_T\right\},\\
\s^1_T &= \max\left\{\rho^1(\ua): \ua\in S^1_T\right\}.
\end{align*}
Then there is a closed nowhere dense definable subset $Z^{rl}_T$ of $T$, such that,
for each $b\in T\backslash Z^{rl}_T$,
\begin{align*}
\s^0_T &= \max\left\{\rho^0(\ua): \ua\in S^0_T\cap\Phi^{-1}(b)\right\},\\
\s^1_T &= \max\left\{\rho^1(\ua): \ua\in S^1_T\cap\Phi^{-1}(b)\right\}.
\end{align*}

Set $\om^{rl}_T = {n+l \choose l} q + \s^1_T - \s^0_T$. Then $\om^{rl}_T = \dim \pi_l(\cR_r(b))$,
for every $b\in T\backslash Z^{rl}_T$. If $l\leq r\leq r'$, then $\om^{r'l}_T \leq \om^{rl}_T$; therefore,
for every $l$, there exists $r(l)= r_T(l)\geq l$ such that $\om^{rl}_T = \om^{r(l)l}_T$, when $r\geq r(l)$. Since
$\Ts$ is finite, we can choose $r(l)$ independent of $T$. This gives the following lemma.

\begin{lemma}\label{lem:rel}
For every $l\in\IN$, there exists $r(l)\in\IN$, $r(l)\geq l$, such that,
if $r>r(l)$ and $T\in \Ts$, then $\pi_l(\cR_r(b)) = \pi_l(\cR_{r-1}(b))$, for every
$b\in T\backslash\left(Z^{rl}_T \cup Z^{(r-1)l}_T\right)$.
\end{lemma}

\begin{rmk}\label{rmk:chev}
This remark will be used in Section \ref{sec:linear}.
Let $b\in X$. We define the  \emph{module of formal relations $\cR(b)=\cR_\infty(b)$ at $b$} as
$$
\R(b):=\left\{W\in\IR\llbracket y\rrbracket^q:  \text{for all } a\in\varphi^{-1}(b),\ \ 
T_aA(x)\,W(\tT_a\varphi(x)) = 0\right\}.
$$ 
Let $l\in\IN$. Then $\pi_l(\R(b)) = \bigcap_{r\geq l} \pi_l(\cR_r(b))$. 
Since $\pi_l(\cR_{r'}(b))\subset\pi_l(\cR_{r}(b))$ when
$l\leq r'\leq r$, it follows that
there exists $r\geq l$ such that $\pi_l(\cR_r(b)) = \pi_l(\cR(b))$ (Chevalley's lemma,
cf. \cite[{\S}II,\,Lemma 7]{Che43}),
and thence that $r(b,l)$ is the least such $r$ .
We call $l\mapsto r(b,l)$ the \emph{Chevalley function at $b$}. 

It is easy to see
that, given $l$,  $r_T(l) \geq r(b,l)$, for all $b\in T\backslash \bigcup_{r\geq r_T(l)}Z^{rl}_T$.
Moreover, if $r(b,l)$ is bounded on $X$, then we can take $r(l) = \max_{b\in X} r(b,l)$.
\end{rmk} 

Now let us fix $r> r(l)$. Let $T\in \Ts$. The diagram of initial exponents $\cN(\cR_r(b))$,
$b\in T\backslash Z^{rl}_T$ takes only finitely many values; choose $b\in T\backslash Z^{rl}_T$
where $\cN(\cR_r(b))$ takes its minimum value. Set $\De(b) = \left(\IN^n \times \{1,\ldots,q\}\right)
\backslash\, \cN(\cR_r(b))$. 

There exists $\ua \in S^0_T \cap \Phi^{-1}(b)$ such that
$\rho^0(\ua) = \s^0_T$. By Corollary \ref{cor:compl}, for any such $\ua$,
\begin{equation*}
\#\De(b) = \rho^0(\ua) = \s^0_T = \rank\left(L_{\alpha,i}^{\beta,j}(a_\nu)\right).
\end{equation*}
Therefore, there exists a nonzero minor
\begin{equation}\label{eqn:minor}
M_T(\ua) = \det\left(L_{\alpha_k,i_k}^{\beta,j}(a_{\nu_k})\right) \quad \left(\text{where } 
(\be, j) \in \De(b),\, k=1,\ldots,\s^0_T\right).
\end{equation}

We can compute $M_T$ for any product coordinate chart containing 
$S^0_T \subset M^s_\phi \subset M^s$. Set
$\Theta^{rl}_T = \{\uc \in S^0_T: M_T(\uc) = 0\}$ and $\Sigma^{rl}_T 
= T\backslash\Phi\left(S^0_T\backslash\Theta^{rl}_T\right)$. Clearly, $\Theta^{rl}_T$ and
$\Sigma^{rl}_T$ are nowhere dense definable subsets of $S^0_T$ and $T$, respectively.
By Cramer's rule, since $M_T(\ua)\neq 0$, the system \eqref{eqn:system} is equivalent
to the system
\begin{equation*}
W_{\ga,k} - \sum_{(\be,j) \notin \De(b)} W_{\be,j} \cdot N^{\be,j}_{T,\ga,k}(\ua)/M_T(\ua) = 0,\quad
(\ga,k)\in\De(b),
\end{equation*} 
where the $N^{\be,j}_{T,\ga,k}(\ua)$ are appropriate minors of the matrix of the
system \eqref{eqn:finitesystem}.

\begin{lemma}\label{lem:diag}
The diagram $\cN(\cR(b'))$ and therefore its complement $\De(b')$ are independent of 
$b' \in T\backslash\left(Z^{rl}_T \cup \Sigma^{rl}_T\right)$.
\end{lemma}

\begin{proof}
Let $(\be^1,j^1)<\cdots<(\be^{t'},j^{t'})$ denote the vertices of $\cN(\cR(b))$, and 
set $t := \max\{h: |\be^h|\le r\}$. Since $(y)^{r+1}\IR\llbracket y\rrbracket^q\subset\R_r(b)$
and $|\be^h|>r$ for $h>t$, the standard basis of $\R_r(b)$ is given by $y^{(\be^h,j^h)},\,h=t+1,\ldots,t'$,
together with
\begin{equation}\label{eq:minor}
G_h(y) := y^{\be^h,j^h}+\sum_{(\gamma,k)\in\Delta(b)}\left(N^{\be^i,j^h}_{T,\gamma,k}(a)/M_T(a)\right)y^{\gamma,k},
\quad h=1,\ldots,t.
\end{equation}

Moreover, for any $\a'\in S^0_T\backslash \Theta^{rl}_T$ such that $\Phi(\a')\notin Z^{rl}_T$ and $i=1,\ldots,r$, we have 
\begin{equation}\label{eqn:diagprime}
y^{\be^h,j^h}+\sum_{(\gamma,k)\in\Delta(b)}\left(N^{\be^i,j^h}_{T,\gamma,k}(a')/M_T(a')\right)y^{\gamma,k}
\in\cR_r(\Phi(\a')).
\end{equation}

Suppose that $b'\in T\backslash\left(Z^{rl}_T \cup \Sigma^{rl}_T\right)$. Since 
$\cN(\cR_r(b))\leq\cN(\cR_r(b'))$, it follows that $(\be_h,j_h) = \exp G_h \in \cN(\cR_r(b'))$, 
for each $h=1,\ldots,t$,
so that $\cN(\cR_r(b))\subset\cN(\cR_r(b'))$, and hence they are equal.
\end{proof}

\subsection{Inductive step of the proof}\label{subsec:ind}
In this subsection, we complete the proof of Theorem \ref{thm:formal}.  By induction
on dimension, it is enough to prove the following.

\begin{prop}\label{prop:ind}
Let $B\subset X$ be a closed definable subset of dimension $d$. Then there exist a closed
definable subset $B'\subset B$, of dimension $<d$, and a function $t:\IN\to\IN$ (where $t(k)\geq k$)
with the following property: if $f:M\to\IR^p$ is $\cC^t$, definable, and $t$-flat on $\varphi^{-1}(B')$
($t=t(k)$), and the system of equations \eqref{eq:varphi2} admits a formal solution to order $t$ at every
$b\in B\backslash B'$, then there exists a $\cC^k$ definable function $g: \IR^n \to \IR^q$ such
that $f - A\cdot(g\circ\vp)$ is $k$-flat on $\varphi^{-1}(B)$.
\end{prop}

\begin{proof}
By Lemma \ref{lem:rreg}, there is a positive integer $\rho$ such that every connected
component of $B$ is $\rho$-regular.

Take any integer $l\geq k\rho$. Let $(\Ss,\Ts)$ denote a definable stratification of $\vp: M\to \IR^n$
such that $\Ss$ is compatible with $B$ and with the covering $\Us$ of $M$ by coordinate charts. 
Note that, if $\P$ is a 
partition of $M$ compatible with $\Us$, and $P$ is an element of $\P$ 
with a boundary point in $U\in \Us$, then $P \subset U$. We use the notation
and results of the previous subsection.

Take $r > r(l)$ (see Lemma \ref{lem:rel}).  Now, taking an appropriate semistratification of $\Phi$,
we can find a finite stratification 
$\{\Lambda_\tau\}$ of $X$, compatible with the sets $T$, $Z^{rl}_T$, $Z^{(r-1)l}_T$, $\Sigma^{rl}_T$,
and a family $\{\Ga_\tau\}$ of connected definable $\cC^\infty$ submanifolds of $M^s_\phi$,
such that, for each $\tau$, $\Phi|_{\Ga_\tau}: \Ga_\tau \to \La_\tau$ is a $\cC^\infty$ definable
submersion, and $\Ga_\tau \subset S^0_T\backslash \Theta^{rl}_T$ whenever $\La_\tau \subset
T\backslash \Sigma^{rl}_T$. 

Note that, for each $\tau$, $\Phi(\overline{\Ga_\tau}\backslash \Ga_\tau)
= \overline{\La_\tau}\backslash \La_\tau$ because $\vp$ is proper and $\Phi|_{\Ga_\tau}: \Ga_\tau \to \La_\tau$ is
a submersion. (This is where the hypothesis that $\vp$ be proper intervenes.)

Suppose that $T\in\Ts$, $\La_\tau\subset T$ and $\dim \La_\tau = \dim T$. Then $\La_\tau$ is
open in $T$ and $\La_\tau \subset T\backslash \left(Z^{rl}_T \cup \Sigma^{rl}_T\right)$.
In this case, we set $\De_\tau = \De_T$ and $M_\tau = M_T$ (using any coordinate chart containing
$S^0_T$).

Let $J := \{\tau: \dim \La_\tau = d\}$ and $B' := \{\La_\tau: \dim \La_\tau < d\}$. For each $\tau\in J$,
$\La_\tau$ lies in some $T$ of dimension $d$, and $M_\tau$ does not vanish on $\Ga_\tau$.
By {\L}ojasiewicz's inequality \cite[Thm.\,6.3]{BM04} there exists a positive integer
$\sigma$ such that 
\begin{equation}\label{eqn:loja}
|M_\tau(\a)|\geq C \dist\left(\a,\,\overline{\Gamma_\tau}\setminus\Gamma_\tau\right)^\sigma,\quad
\ua\in\Gamma_\tau,
\end{equation}
(where the constant $C$ is local) for all $\tau\in J$, in any coordinate chart $U$ which
intersects $\overline{\Gamma_\tau}\setminus\Gamma_\tau$ (and therefore contains $S^0_T$), 
where $\La_\tau\subset T$.

We will show that, to get a function $t=t(k)$ as required in Proposition \ref{prop:ind},
we can take any integer $t \geq r + \s$; i.e., we can take $t(k) \geq r + \s$, where 
\begin{enumerate}
\item[(1)] $\s$ is the {\L}ojasiewicz exponent of \eqref{eqn:loja}, 
\item[(2)] $r > r(l)$ (as given by Lemma \ref{lem:rel}),
\item[(3)] $l \geq k\rho$, where $\rho$ is the exponent of regularity of $B$. 
\end{enumerate}
See also the sketch of the proof of Theorem \ref{thm:linear}.

Take $t\geq r+\sigma$ and suppose that $f$ satisfies the hypotheses of Proposition 
\ref{prop:ind}. In particular, for all $b\in B$, 
$$
T^t_af(x) \equiv T^t_aA(x)\, W_b(\tT^t_a\vp(x))\quad \text{mod } (x)^{t+1}\IR\llb x\rrb^p,\quad \text{fo all } a\in \vp^{-1}(b),
$$
where $W_b \in \IR[y]^q$. 

Let $\tau\in J$. For each $b\in\Lambda_\tau$, there is a unique polynomial 
$$V_\tau(b,y)=\sum_{(\beta,j)\in\Delta_\tau}V^{\beta,j}_\tau(b)y^{\beta,j}\in\IR[y]^q,
$$
of degree $\leq r$, such that 
$$
W_b(y)-V_\tau(b,y)\in\R_r(b) \quad \text{and} \quad \supp V_\tau(b,y) \subset \De_\tau.
$$

\begin{lemma}\label{lem:loj}
Each coefficient $V^{\beta,j}_\tau(b)$
tends to zero as $b$ tends to a point
$b_0\in\overline{\Lambda_\tau}\backslash\Lambda_\tau$.
\end{lemma}

\begin{proof}
For all $a\in\varphi^{-1}(b)$, 
$$
T^{r}_af(x)\equiv T^{r}_aA(x)\,V_\tau\left(b,\tilde T^{r}_a\varphi(x)\right)\quad \mod 
(x)^{r+1}\IR\llbracket x\rrbracket^p.
$$
Write $f = (f_1,\ldots,f_p)$. Let $a_0 \in \vp^{-1}(b_0) \cap \left(\overline{\Ga}_\tau\backslash \Ga_\tau\right)$.
In any coordinate chart $U$ containing $a_0$,
\begin{equation}\label{eqn:cramer2}
\frac{\p^{|\alpha|} f_i}{\p x^\al} (a_\nu)=\sum_{(\beta,j)\in\Delta_\tau}L_{\alpha,i}^{\beta,j}(a_\nu)\,
V^{\beta,j}_\tau(\Phi\left(\a\right)),
\end{equation}
for all $a\in\Gamma_\tau,\,|\alpha|\leq r,\,i=1,\ldots,p$, and $\nu=1,\ldots,s$.

By Cramer's rule, 
\begin{equation}\label{eq:cramer3}
V^{\beta,j}_\tau\left(\Phi\left(\a\right)\right)=H^{\beta,j}_\tau\left(\a\right)/M_\tau\left(\a\right),
\end{equation}
where $H^{\beta,j}_\tau$ is $\cC^\sigma$ and $\sigma$-flat on $\overline{\Gamma_\tau}\setminus\Gamma_\tau\subset\Phi^{-1}(\overline{\Lambda_\tau}\setminus\Lambda_\tau)\subset\Phi^{-1}(B')$. The assertion
follows from \eqref{eqn:loja}.
\end{proof}

\begin{lemma}\label{lem:GtauWF}
For each $\tau\in J$, set
$$
G_\tau(b,y) := \pi_l\left(V_\tau(b,y)\right) = \sum_{\substack{(\beta,j)\in\Delta_\tau\\|\beta|\le l}}V^{\beta,j}_\tau(b)\,y^{\beta,j}.
$$
Then $G_\tau$ is a definable $\cC^l$ Whitney field on $\Lambda_\tau$.
\end{lemma}

\begin{proof}
We have to show that
$$
D_{b,v} G_\tau^{l-1}(b,y) = D_{y,v} G_\tau(b,y),\quad \text{for all }\ b\in\La_\tau,\, v\in T_b\La_\tau
$$
(see Lemma \ref{lem:borel}).

Let $\Psi_\nu: M^s \to \IR^n$ and $h_\nu: M^s \to \IR^p$, $\nu=1,\ldots,s$, denote the mappings
$\Psi_\nu(\ua) = \Psi_\nu(a_1,\ldots,a_s) := \vp(a_\nu)$ and $h_\nu(\ua) := f(a_\nu)$. 
Likewise let $A_\nu(\ua) := A(a_\nu)$. Then
$\Psi_\nu|_{M^s_\vp} = \Phi$ and, for each $\ua\in M^s$, the polynomial $T^r_{\ua}\Psi_\nu(\ux)$
can be identified with $T^r_{a_\nu}\vp(x)$.

Let $b\in\La_\tau$. Choose $\ua\in\Phi^{-1}(b)$ such that, for any $W\in\IR[y]^q$, $W\in\cR_{r-1}(b)$
if and only if $W\left(\tT^r_{a_\nu}\vp(x)\right) \equiv 0 \mod (x)^r\IR\llb x\rrb^q$, $\nu=1,\ldots,s$. Take $S\in\Ss^s$
such that $\ua\in S$. Then $\La_\tau$ is open in $T = \Phi(S)$. We have
\begin{equation}\label{eq:h}
T^r_{\uc}h_\nu(\ux) = T^r_{\uc}A_\nu(\ux)\cdot V_\tau\left(\Phi(\uc), \tT^r\Psi_\nu(\uc,\ux)\right) \mod (\ux)^{r+1}\IR\llb\ux\rrb^p.
\end{equation}
for all $\uc\in S\cap \Phi^{-1}(\La_\tau)$ and $\nu=1,\ldots,s$. 

By \eqref{eq:cramer3}, the $V^{\be,j}_\tau$ are
$\cC^1$ (in fact, $\cC^\s$) on $\La_\tau$.
Let $v$ be any tangent vector to $\La_\tau$ at $b$.
Since $\Phi|_S: S\to T$ is a submersion, there exists $u\in T_{\ua} S$ such that 
$d_{\ua}\left(\Phi|_S\right)(u)=d_{\ua}\Psi_\nu(u) =v$. 
By Lemma \ref{lem:borel} and \eqref{eq:h}, 
$$
T^r_{\ua}A_\nu(\ux)\cdot\left(D_{b,v} V^{r-1}_\tau\left(b,\tT^r_{\ua}\Psi_\nu(\ux)\right)-D_{y,v}V_\tau\left(b,\tT^r_{\ua}\Psi_\nu(\ux)\right)\right)\equiv0\mod (\ux)^r\IR\llbracket\ux\rrbracket^q
$$
where $V^{r-1} := \sum_{|\be|\leq r-1}V^{\beta,j}_\tau(b)y^{\beta,j}$; hence
$$
T^r_{a_\nu}A(x)\cdot\left(D_{b,v} V^{r-1}_\tau\left(b,\tT^r_{a_\nu}\varphi(x)\right)-D_{y,v}V_\tau\left(b,\tT^r_{a_\nu}\varphi(x)\right)\right)\equiv0\mod (x)^r\IR\llbracket x\rrbracket^q.
$$
It follows that 
$$
D_{b,v} V^{r-1}_\tau\left(b,y\right)-D_{y,v}V_\tau\left(b,y\right) \in\R_{r-1}(b),
$$
and, by Lemma \ref{lem:rel}, that
$$
D_{b,v}G^{l-1}_\tau(b,y)-D_{y,v}G_\tau(b,y)\in \pi_{l-1}(\cR_r(b))=\pi_{l-1}(\cR_{r-1}(b)).
$$
On the other hand, $$
\supp \left(D_{b,v}G^{l-1}_\tau(b,y)-D_{y,v}G_\tau(b,y)\right) \subset \Delta_\tau;
$$ 
therefore, $D_{b,v}G^{l-1}_\tau(b,y)=D_{y, v}G_\tau(b,y)$, as required.
\end{proof}

We can now complete the proof of Proposition \ref{prop:ind}.
Define $G\in\cC^0(B)[y]$ by $G:=G_\tau$ on $\Lambda_\tau\subset B\backslash B'$ and $G:=0$ otherwise. 
By  Lemma \ref{lem:hestenes}, since $B$ is $\rho$-regular, $G$ induces a $\cC^k$-Whitney field on $B$
(truncating terms of degree $> k$). We obtain a $\cC^k$ definable function $g:\IR^n\to \IR^q$, as required, using
Theorem \ref{thm:ominWhitney}.
\end{proof}

\section{Linear loss of differentiability}\label{sec:linear}
In this section, we show that Theorem \ref{thm:varphi} in the special case that $\vp$ is the
identity mapping of $\IR^n$ (i.e., Theorem \ref{thm:main} in the case that $A(x)$ 
is $\cC^\infty$ and definable in a polynomially bounded $o$-minimal structure) holds with
linear loss of differentiability.

\begin{thm}\label{thm:linear}
Let $A(x)$ denote a $p \times q$ matrix whose entries are functions on $\IR^n$ that are
$\cC^\infty$ and semialgebraic (or definable in a polynomially bounded $o$-minimal structure).
Then there is a function $r(m) = \la m + \mu$, for all $m\in\IN$, where $\la, \mu\in\IN$,
such that, for all $m\in \IN$, if $F:\IR^n\to\IR^p$ is semialgebraic (or definable) and the
system of equations
\begin{equation}\label{eq:linear}
A(x)G(x) = F(x)
\end{equation}
admits a $\cC^{r(m)}$ solution $G(x)$, then there is a semialgebraic
(or definable) $\cC^m$ solution.
\end{thm}

In Lemma \ref{lem:unifchev} below, we show that, under the hypotheses of Theorem \ref{thm:linear}, 
there is a uniform linear bound on the Chevalley function $r(a,l)$, $a\in\IR^n$ (recall Remark \ref{rmk:chev}); 
i.e., $r(a,l)\leq \xi l + \eta$, for all $a\in\IR^n$, where $\xi, \eta\in\IN$.

Our proof of Theorem \ref{thm:varphi}, however, does not immediately provide a linear function
$r(m)$ under the assumption that the Chevalley function has a uniform linear bound, because
the exponent of the {\L}ojasiewicz inequality \eqref{eqn:loja} involved in $r(m)$,
depends on a stratification associated to the order of differentiability $k$, and on
the minors $M_T$ associated to strata $T$, in a way that is not \emph{a priori}
given by a linear function of $k$.

In the proof of Theorem \ref{thm:linear}, we use a finite definable stratification of $\IR^n$ with the property that
the diagrams of initial exponents of both the module of (vector-valued) 
formal power series at a point $a$ generated
by (the formal Taylor expansions of) the columns of the matrix $A$, and the module of formal relations $\cR(a)$
among the columns (cf. Remark \ref{rmk:chev}), 
are constant on strata. (See \cite[Thm.\,5.1]{BM17}; it is easy to see that the proof of the latter
gives a finite definable stratification under the hypotheses of Theorem \ref{thm:linear}.) We then need
{\L}ojasiewicz inequalites for the initial coefficients of generators of such
(parametrized families of) modules on a stratum, which can be compared to the way that the minor $M_T$
appears in \eqref{eq:minor}. 

Our proof of Theorem \ref{thm:linear} then follows the argument of \S\ref{subsec:ind}, in the setup
(provided by the stratification theorem \cite[Thm.\,5.1]{BM17} as above) 
of the proof of \cite[Thm.\,1.1]{BM17} (the latter is a $\cC^\infty$ version of Theorem \ref{thm:linear});
for this reason, we only give a sketch below, and leave full details to the reader.

\begin{rmk}\label{rmk:infty}
For equations of the form \eqref{eq:varphi}, in general, the following conditions are equivalent:
\begin{enumerate}
\item[(1)] there is a uniform bound on the Chevalley functions $r(b,l)$, $b\in\varphi(M)$;
\item[(2)] there is a definable stratification of $\IR^n$ such that the diagram of initial exponents
of the module of formal relations $\cR(b)$ is constant on strata;
\item[(3)] the conclusion of Theorem \ref{thm:formal} in the $\cC^\infty$ case holds.
\end{enumerate}
This is proved in \cite{BM98} in the case that $A(x)$ and $\vp(x)$ are real-analytic, and
is extended to the case that these functions are $\cC^\infty$ and definable, in \cite{Bib}.

These equivalent conditions do not hold, however, in general, under the hypotheses of Theorem
\ref{thm:formal}, even in the special case of the composite function problem (the case
that $A(x) = I$), for a proper real-analytic mapping $\varphi$; there is a counterexample
due to Paw{\l}ucki \cite{Paw89}. The equivalent conditions above do hold in
the case that $A$ and $\varphi$ are semialgebraic, and in the case of the composite function
problem where $\vp$ is real-analytic and \emph{regular} in the sense of Gabrielov \cite{Gab73};
see \cite{BM87-2}. In the latter case, moreover, the Chevalley function has a uniform
linear bound \cite{ABM}.
\end{rmk}

\begin{lemma}\label{lem:unifchev}
Under the hypotheses of Theorem \ref{thm:linear}, the Chevalley function admits
a uniform linear upper bound.
\end{lemma}

\begin{proof}
For each $a\in\IR^n$, let $\cA(a,x) \subset \IR\llb x \rrb^p$ denote the submodule generated
by the formal expansions at $a$ of the columns of $A(x)$. According to \cite[Thm.\,5.1]{BM17},
there is a finite definable stratification of $\IR^n$ such that the diagram of initial exponents
$\cN(\cA(a,x))$ (as a function of $a$) is constant on strata. The assertion follows from Corollary \ref{cor:AR}.
\end{proof}

\begin{proof}[Sketch of the proof of Theorem \ref{thm:linear}]
We follow the argument of the proof of Proposition \ref{prop:ind}, except that we use
a finite definable stratification compatible with $B' \subset B$ (as in Proposition \ref{prop:ind})
and with the stratification by the diagrams of initial exponents provided by \cite[Thm.\,5.1]{BM17}. 

As in Proposition \ref{prop:ind}, we define a function $t=t(k)$ which has three main ingredients: we can
take $t\geq r + \s$, where $r\geq r(l)$, $l\geq k\rho$, and
\begin{enumerate}
\item[(1)] $r(l)$ is given by a uniform linear bound on the Chevalley function;
\item[(2)] $\rho$ is a bound on the exponents of regularity of the connected components of $B$
(as in the proof of Proposition \ref{prop:ind});
\item[(3)] $\s = \s(k)$ is a linear function of $k$ given by bounds on {\L}ojasiewicz exponents.
\end{enumerate}
Once we have these ingredients, we can argue in a way that is very similar to the proof
of Proposition \ref{prop:ind}.

Let us give a rough idea of (3). Let $\cA(a,x) \subset \IR\llb x\rrb^p$ denote the module
generated by the formal Taylor expansions at $a\in\IR^n$
of the columns of $A$, and let $\cR(a,x) \subset \IR\llb x\rrb^q$ denote the module of relations
among the formal expansions at $a$ of the columns of $A$. The stratification theorem
\cite[Thm.\,5.1]{BM17} provides families parametrized by the points $a$ of every stratum, of
generators of $\cA(a,x)$ and $\cR(a,x)$. These generators come from the formal division
algorithm (Theorem \ref{thm:formaldiv} and Corollary \ref{cor:diag}, where the coefficient
ring $A$ is a ring of $\cC^\infty$ definable functions):
the coefficients of the generators (as functions of $a$) are quotients of $\cC^\infty$ definable
functions whose denominators are products (of powers) of the initial coefficients of representatives of the
vertices of the diagrams, and the number of such factors in the denominator of a coefficient of a
monomial $x^\be$, $|\be|\leq k$, is at most $k$. So we get a linear function $\s(k) = \zeta k$,
where $\zeta$ is a bound on the {\L}ojasiewicz exponents of the initial coefficients, for all strata.
The reader can check the details by comparing with the proof of \cite[Thm.\,1.1]{BM17}.
\end{proof}


\end{document}